\documentclass[10pt]{article}
\usepackage{amssymb,amsmath,amsthm,latexsym}
\usepackage{amsfonts}
\usepackage[twlrm]{rawfonts}
\usepackage[english]{babel}
\usepackage[latin1]{inputenc}
\usepackage{fancyhdr}
\usepackage{indentfirst}
\usepackage{color}
\everymath{\displaystyle}

\newtheorem{theorem}{Theorem}[section]
\newtheorem{cor}{Corolary}[section]
\newtheorem{lemma}{Lemma}[section]
\newtheorem{pro}{Proposition}[section]
\newtheorem{claim}{Claim}[section]

\newcommand{\R}{\mathbb{R}}

\newcommand{\ddd}{\displaystyle}

\begin{document}
	
\setlength{\baselineskip}{6.5mm} \setlength{\oddsidemargin}{8mm}
\setlength{\topmargin}{-3mm}

\title{\bf On an Ambrosetti-Prodi type problem in $\R^N$}

\author{Claudianor O. Alves\thanks{C.O.Alves was partially supported by CNPq/Brazil 304804/2017-7, coalves@mat.ufcg.edu.br}\, , Romildo N. de Lima\thanks {romildo@mat.ufcg.edu.br}\, and \,Alânnio B. Nóbrega \thanks{alannio@mat.ufcg.edu.br}\,\,\,\,\,\,\vspace{2mm}
	\and {\small  Universidade Federal de Campina Grande} \\ {\small Unidade Acadêmica de Matemática} \\ {\small CEP: 58429-900, Campina Grande - PB, Brazil}\\}

\date{}

\maketitle

\normalsize

\begin{abstract}
In this paper we study results of existence and non-existence of solutions for the following Ambrosetti-Prodi type problem  
$$
\left\{
\begin{array}{lcl}
-\Delta u=P(x)\Big( g(u)+f(x)\Big) \mbox{ in } \mathbb{R}^N,\\
u \in D^{1,2}(\R^N),\ \lim_{|x|\to +\infty}u(x)=0,
\end{array}
\right.
\eqno{(P)}
$$
where $N\geq3$, $P\in C(\R^N,\R^+)$, $f\in C(\R^N)\cap L^{\infty}(\R^N)$ and $g\in C^1(\R)$. The main tools used are the sub-supersolution method and Leray-Schauder topological degree theory.

\vspace{0.5cm}


\end{abstract}
	
\section{Introduction and main results}

This paper concerns with the existence and non-existence of solutions for the following Ambrosetti-Prodi type problem  
$$
\left\{
\begin{array}{lcl}
-\Delta u=P(x)\Big( g(u)+f(x)\Big) \mbox{ in } \mathbb{R}^N,\\
u \in D^{1,2}(\R^N),\ \lim_{|x|\to +\infty}u(x)=0,
\end{array}
\right.
\eqno{(P)}
$$
where $N\geq3$, $P\in C(\R^N,\R^+)$, $f\in C(\R^N)\cap L^{\infty}(\R^N)$ and $g\in C^1(\R)$. 

The main motivation to study the problem $(P)$ comes from the seminal paper by Ambrosetti and Prodi  \cite{Ambrosetti-Prodi} that studied the existence and non-existence of solution for the problem
\begin{equation} \label{P1}
	\left\{                   
	\begin{array}{lcl}
	-\Delta u= g(u)+f(x),\mbox{ in } \Omega,\\
	u=0,\mbox{ in } \partial\Omega,
	\end{array}
	\right.
\end{equation}
where $\Omega\subset \R^N$ with $N \geq 3$, is a bounded domain, $g$ is a $C^2-$function with 
$$
g''(s)>0, \quad \forall s \in \R \quad \mbox{and} \quad 0<\lim_{s\rightarrow -\infty}g'(s)<\lambda_1<\lim_{s\rightarrow \infty}g'(s)< \lambda_2.
$$ 
In order to prove their results, Ambrosetti and Prodi used a global result of inversion to proper functions to show the existence  of a closed manifold $M$ dividing the space $C^{0,\alpha}(\Omega)$ in two connected components $O_1$ and $O_2$ such that: \\
     \noindent (i) If $f$ belongs to $O_1$,  the problem (\ref{P1}) has no solution;\\
	 \noindent (ii) If $f$ belongs to $M$,  the problem (\ref{P1}) has exactly one solution;\\
	 \noindent (iii) If $f$ belongs to $O_2$,  the problem (\ref{P1}) has exactly two solution;

In \cite{Berger-Podolak}, Berger and Podolak proposed the decomposition of function $f$ in the form $f=t\phi+f_1,$ where $\phi$ is eigenfunction associated to first eigenvalue of "$-\Delta$"
\begin{equation} \label{P2}
\left\{                   
\begin{array}{lcl}
-\Delta u= g(u)+t\phi+f_1,\mbox{ in } \Omega,\\
u=0,\mbox{ in } \partial\Omega,
\end{array}
\right.
\end{equation}
then using the Liapunov-Schmidt method they showed the existence of $t_0\in \R$ such that (\ref{P2}) has at least two solutions if $t<t_0$, at least one solution if $t=t_0$ and no solutions if $t>t_0$. Ever since many papers have dealt with this theme, we cite the works by Brézis and Turner \cite{Brezis-Turner}, de Figueiredo and Solimini \cite{djairo2}, de Figueiredo and Yang \cite{djairo 3}, Mawhin \cite{M} de Morais Filho \cite{de Morais Filho} and de Figueiredo \cite{Djairo}.   More recently, de Figueiredo and Sirakov \cite{djairo 4} studied a Ambrosetti-Prodi problem for an operator in non-divergence form, Aizicovici, Papageorgiou and Staicu \cite{Aizicovici-Papageorgiou-Staicu} and Arcoya and Ruiz \cite{Arcoya-Ruiz}  treated the quasilinear operator cases, de Paiva and Montenegro studied a quasilinear Newman problem \cite{PM} and   Presoto and de Paiva \cite{Presoto} showed an Ambrosetti-Prodi type result to a Newmann problem with a gradient non-linearity.  

Although that subject has been studied in the most varied situations, we did not find in the literature articles about Ambrosetti-Prodi type  problems in whole space $\R^N$ and this has motivated the present paper. In order to get our main results, we face some difficulties, because some estimates do not follow as in bounded domain case, for example in all of the papers mentioned above involving Dirichlet boundary conditions, the fact that the  outer normal derivate on the boundary is positive, see \cite[Lemma 3.4]{G-T}, is a key point to prove some estimates. Here, we overcome these difficulties by adapting for our case some ideas found in  \cite{Alves-delima-Nobrega,Alves-delima-Nobrega2}, where one of the main points  was to fix a suitable function space where the topological degree could be used to establish a second solution for our case, for more details see Section $5$.  The reader is also invited to see that the arguments used to prove a priori estimates for our problem is a little bit different from those found in the literature because we are working in whole $\mathbb{R}^N$, see Section 4.  

Before stating our main results, we need to fix the assumptions on the functions $P$ and $g$. In the sequel, $g: \R \rightarrow \R$ is a $C^1-$function that satisfies the following inequalities

$$
\limsup_{s \to -\infty}\frac{g(s)}{s}<\lambda_1<\liminf_{s \to \infty}\frac{g(s)}{s} \leqno{(G_1)}
$$

Hereafter, the limits above can be infinite, where  $\lambda_1$ is the first eigenvalue to following problem
$$
\left\{
\begin{array}{lcl}
-\Delta u=\lambda P(x)u \mbox{ in } \mathbb{R}^N,\\
u \in D^{1,2}(\R^N),\ \lim_{|x|\to +\infty}u(x)=0,
\end{array}
\right.
\eqno{(P)_{\lambda}}
$$
which has the variational characterization below
\begin{equation} \label{lambda1}
\lambda_1=\inf_{v\in D^{1,2}(\R^N)\setminus\{0\}}\left\{\frac{\int_{\R^N}|\nabla v|^2dx}{\int_{\R^N}P(x)|v|^2dx}\right\}.
\end{equation}
Related to the $\lambda_1$, we have an eigenfunction $\phi_1$ that satisfies 
\begin{equation}\label{m1}
0<C_1\leq |x|^{N-2}\phi_1(x)\leq C_2, \quad \forall x\in\R^N,
\end{equation}
for positive constants $C_1,C_2$. For more details about this subject see \cite{Alves-Lima-Souto}.

As an immediate consequence of the $(G_1)$, there exist positive constants $\Theta$, $\underline{\mu}$ and $\overline{\mu}$ such that $\underline{\mu}<\lambda_1<\overline{\mu},$ 
\begin{equation}\label{eq01}
g(s)\ge \underline{\mu}s-\Theta,\ \forall s \in \R,	
\end{equation}
and
\begin{equation}\label{eq02}
g(s)\ge \overline{\mu}s-\Theta,\ \forall s \in \R.	
\end{equation}

Related to the function $P:\mathbb{R}^N \to \mathbb{R}^+ $, we consider that it is a continuous function satisfying: 

\noindent $(P_1)$ $|\cdot|^2 P(\cdot) \in L^{1}(\R^N)\cap L^{\infty}(\R^N)$;

\noindent $(P_2)$ $\int_{\mathbb{R}^N}\frac{P(y)}{|x-y|^{N-2}}dy \le \frac{C}{|x|^{N-2}}$, for all $x\in\R^N \setminus \{0\}$, for some $C>0$.

We denote by $\mathcal{N}$ the eigenspace associated with the first eigenvalue $\lambda_1$. By \cite{Alves-Lima-Souto}, it is well known that $dim \, \mathcal{N}=1$, then we can assume that $\mathcal{N}=Span\{\phi\}$, where $\phi$ is one positive eigenfunction associated with $\lambda_1$ with 
\begin{equation}
\int_{\R^N}P(x)|\phi|^2dx=1.
\end{equation}
Hence, we can write $f=t\phi+f_1,$ where $f_1\in C(\R^N)\cap L^{\infty}(\R^N)$ with
\begin{equation}
\int_{\R^N}P(x)f_1\phi dx=0\ \mbox{and}\ \int_{\R^N}P(x)f\phi dx=t.
\end{equation}
From this,  problem $(P)$ can be rewritten as follows
$$
\left\{
\begin{array}{lcl}
-\Delta u=P(x)\Big( g(u)+t\phi(x)+f_1(x)\Big) \mbox{ in } \mathbb{R}^N,\\
u \in D^{1,2}(\R^N),\ \lim_{|x|\to +\infty}u(x)=0.
\end{array}
\right.
\eqno{(\tilde{P})}
$$

Our first result is the following:
\begin{theorem}\label{T1}
	Assume the conditions $(G_1)$, $(P_1)$ and $(P_2)$. Then, for each $f_1 \in \mathcal{N}^{\perp}$ there is a number $\alpha(f_1)$ such that:
	\begin{description}
		\item[(i)] The problem $(\tilde{P})$ has no solution whenever $t>\alpha(f_1)$;
		\item[(ii)] If $t<\alpha(f_1)$, then $(\tilde{P})$ has at least one solution. 
	\end{description}
\end{theorem} 

Our second result is the following:
\begin{theorem}\label{T2}
	Assume the conditions $(G_1)$, $(P_1)$ and $(P_2)$. Moreover, assume that $g$ is an increasing function satisfying
	\begin{equation} \label{sigma}
	\lim_{s\rightarrow +\infty}\frac{g(s)}{s^{\sigma}}=0,
	\end{equation}
	where $\sigma=\frac{N}{N-2}$. Then, for each $f_1 \in \mathcal{N}^{\perp}$ there is a number $\alpha(f_1)$ such that:
	\begin{description}
		\item[(i)]  If $t<\alpha(f_1)$, then $(\tilde{P})$ has at least two solutions; 
		\item[(ii)] If $t=\alpha(f_1)$, then $(\tilde{P})$ has at least one solution.
	\end{description}
\end{theorem} 

The paper is organized as follows: In Section $2$  we present some preliminaries results that play important role throughout the work. In Section $3$ we ensure the existence of sub and super-solution to $(P)$, and we use the sub-super solution method to prove the Theorem \ref{T1}. In section $4$ we obtain a priori estimate to the solutions of $(P)$, while in Section $5$ we define the solution operator associated to $(P)$ and study some properties of this operator. Finally, in Section $6$, we prove the Theorem \ref{T2}.

\noindent \textbf{Notations}
\begin{itemize}

	\item $B_{r}(x)$ denotes the ball centered at the $x$ with radius $r>0$ in $\R^N$.
	
	\item $L^{s}(\R^N)$, for $1\leq s\leq\infty$, denotes the Lebesgue space with the usual norm $|u|_s$. 
	
	\item $L^{2}_{H}(\R^N)$ denotes the class of real valued Lebesgue measurable functions $u$ such that
	$$
	\int_{\R^N}H(x)|u(x)|^{2}dx<\infty.
	$$
	$L^{2}_{H}(\R^N)$ is a Hilbert space endowed with the inner product
	$$
	(u,v)_{2,H}=\int_{\R^N}H(x)u(x)v(x)dx, \quad \forall  u,v\in L^{2}_{H}(\R^N).
	$$
	The norm associated with this inner product will be denoted by $|\cdot|_{2,H}$.
	
	\item $D^{1,2}(\R^N)$ denotes the Sobolev space endowed with inner product
	\begin{equation*}
	(u,v)_{1,2}=\int_{\mathbb{R}^N}\nabla u\nabla v dx,\quad u,v\in D^{1,2}(\R^N).
	\end{equation*}
	The norm associated with this inner product will be denoted by $|\cdot|_{1,2}$.

	\item $H^{1}(\R^N)$ denotes the Sobolev space endowed with inner product
	\begin{equation*}
	(u,v)_{H^1}=\int_{\mathbb{R}^N}\nabla u\nabla v dx+\int_{\mathbb{R}^N} u v dx,\quad u,v\in H^{1}(\R^N).
	\end{equation*}
	The norm associated with this inner product will be denoted by $|\cdot|_{H^1}$.
	
	\item We denote by $E$ the Banach space given by
	\begin{equation*}
	E:=\{u\in C(\R^N);\quad \sup_{x\in\R^N}|u(x)|<\infty\}
	\end{equation*}
	endowed with the norm $|\cdot|_{\infty}$.
	
	\item If $u$ is a mensurable function, we denote by $u^{+}$ and $u^{-}$  the positive and negative part of $u$ respectively, which are given by
	$$
	u^{+}=\max\{u,0\} \quad \mbox{and} \quad u^{-}=\min\{u,0\}.
	$$	
	
	\item $C,C_1,\cdots,C_n$ are positive constants.
	
\end{itemize}

\section{Preliminaries Results}

We start this section proving a version of sub and super-solution method to the problem $(P)$. We say $u \in D^{1,2}(\R^N)$ is a sub-solution to problem $(P)$ if $\lim_{|x| \to +\infty}u(x)=0$ and
$$\int_{\R^N}\nabla u \nabla \varphi dx \le \int_{\R^N}P(x)\left(g(u)+f(x)\right)\varphi dx,\quad \forall  \varphi\in C_0^{\infty}(\R^N) \quad \mbox{and} \quad \varphi\ge 0. $$
Similarly $u \in D^{1,2}(\R^N)$ is a super-solution to $(P)$ if the above reverse inequality is valid.
 
 \begin{pro}\label{p2}
 	Suppose that problem $(P)$ has a sub-solution $\underline{u}$ and a super-solution $\overline{u}$, with $-\infty<\underline{c}\le \underline{u} \le \overline{u}\le \overline{c}<+\infty.$ Then problem $(P)$ has a solution $u\in D^{1,2}(\R^N)$ such that $\underline{u} \le u\le \overline{u}$. 
 \end{pro}
 \begin{proof}
 	Firstly, we fix $h(x,s)=g(s)+f(x)$ and define the function 
 $$
 \tilde{h}(x,t)=\left \{\begin{array}{lcl}
 	h(x,\overline{u}), \ \mbox{if}\ t \ge \overline{u}(x),\\
 	h(x,t), \ \mbox{if}\  \underline{u}(x)<t<\overline{u}(x),\\
 	h(x,\underline{u}), \ \mbox{if}\ t \le \underline{u}(x).
 	\end{array} \right.
$$ 
Using the function $\tilde{h}$, we fix the auxiliary problem below 	
$$
\left \{\begin{array}{lcl}
 -\Delta u=P(x)\tilde{h}(x,u),\ \mbox{in}\ \R^N\\
 u\in D^{1,2}(\R^N).
 \end{array} \right.
\eqno{(P)_A}
$$	
Associated with $(P)_A$, we have the energy functional $I:D^{1,2}(\R^N) \rightarrow \R$ given by 
$$
I(u)=\frac{1}{2}\int_{\R^N}|\nabla u|^2dx-\int_{\R^N}P(x)\tilde{H}(x,u)dx,
$$
where $\tilde{H}(x,t)=\int\limits_{0}^{t}\tilde{h}(x,s)ds$. 
 
 	Using Minimization Methods, it is easy to see that $I$ has a global minimum $u \in D^{1,2}(\mathbb{R}^N)$, which must satisfy
 $$
 \int_{\R^N}\nabla u \nabla \varphi dx=\int_{\R^N}P(x)\tilde{h}(x,u)\varphi dx, \quad \forall \varphi \in D^{1,2}(\R^N).
 $$
Considering $\varphi=(\underline{u}-u)^+,$ we obtain
 	$$\int_{\R^N}\nabla u \nabla (\underline{u}-u)^+ dx=\int_{\R^N}P(x)h(x,\underline{u})(\underline{u}-u)^+ dx\ge \int_{\R^N}\nabla \underline{u} \nabla (\underline{u}-u)^+ dx,$$ 
and so, 
 $$
 \int_{\mathbb{R}^N}|\nabla \underline{u}-\nabla u|^{2} \, dx \leq 0,
 $$
 from where it follows that $\underline{u}\le u$. Analogously, we can ensure that $u\le \overline{u}$, and thus $u$ is a solution to $(P)$ with $\underline{u} \le u\le \overline{u}$.
 \end{proof}
%
Now, we will study the regularity of solutions of the problem 
$$
\left\{ \begin{array}{lcl}
-\Delta u=P(x)h(x,u),\ \mbox{in}\ \R^N,\\
u\in D^{1,2}(\R^N),
\end{array}
\right.\eqno(Q)
$$
where $h$ is a sub-critical function. This result plays an important role throughout the article.
\begin{lemma}\label{reg}
Let $u \in D^{1,2}(\R^N)$ be a solution to $(Q)$ with $h:\R^N\times\R \rightarrow \R$ being  a $C^1-$function satisfying
$$
|h(x,t)|\leq C(1+|t|^p), \quad \forall (x,t) \in \mathbb{R}^N \times \mathbb{R}, 
$$
for some $p \in (1,2^*-1)$. Then $u \in C(\R^N)\cap L^\infty(\R^N)$.
\end{lemma}
\begin{proof}
Since $u \in D^{1,2}(\R^N)$, we have $u \in L^{2^*}(\R^N)$, and thus, $u\in L^{s}(B)$ for each ball $B \subset \R^N$ and $s \in [1,2^*].$ Arguing as in \cite[Proposition 2.15]{Rabinowitz}, we deduce that $u \in C(\R^N)$ and there is a continuous $\Upsilon:\R \rightarrow \R$ such that 
\begin{equation}\label{eq03}
	\|u\|_{C(\overline{B_1(z)})}\le \Upsilon(|u|_{L^{2^*}(B_2(z))}), \ \mbox{for all}\ z \in \R^N.
\end{equation}
Since $u \in L^{2^*}(\R^N),$ given $\delta>0$ there exists $R>0$ such that 
$$|u|_{L^{2^*}(B_2(z))}<\delta,\ \mbox{for }|z|\geq R.$$
As $\Upsilon$ is a continuous function, given $\varepsilon>0$ there is $\delta > 0$ such that
$$
\Upsilon(|u|_{L^{2^*}(B_2(z))})<\Upsilon(0)+\varepsilon,\ \mbox{whenever}\ |u|_{L^{2^*}(B_2(z))}<\delta.
$$
This combined with (\ref{eq03}) gives
$$
u(z)\le \|u\|_{C(\overline{B_1(z)})}\le \Upsilon(|u|_{L^{2^*}(B_2(z))})\le \Upsilon(0)+\varepsilon, \quad \mbox{for} \quad |z|\geq R,
$$
showing that $u \in L^{\infty}(\R^N).$
\end{proof}

Before finishing this section, we present the following comparison principle .  
\begin{lemma}\label{max}
	Let $u\in D^{1,2}(\R^N)$ be a function verifying
	\begin{equation}\label{Imax}
	-\Delta u-\underline{\mu}P(x)u\ge 0,\mbox{ in } \mathbb{R}^N.
	\end{equation}
Then, $u\ge0$.
\end{lemma}
\begin{proof}
	Since $u\in D^{1,2}(\R^N)$ verifies (\ref{Imax}), then
$$
\int_{\R^N}\nabla u \nabla \phi dx - \int_{\R^N}\underline{\mu}P(x)u\phi dx \ge 0, \quad 
$$
for all $\phi\in D^{1,2}(\R^N)$ with $\phi\ge0$ in $\R^N$. If $u^-\neq0$,  considering $\phi=-u^-$, we derive that
	$$-\int_{\R^N}|\nabla u^-|^2dx + \int_{\R^N}\underline{\mu}P(x)|u^-|^2 dx \ge 0,$$
	or yet,
	$$\underline{\mu}\ge \frac{\int_{\R^N}|\nabla u^-|^2dx}{\int_{\R^N}P(x)|u^-|^2 dx }\ge \lambda_1,$$
	which contradicts the hypothesis $\underline{\mu}<\lambda_1.$ Thereby, $u=u^+\ge 0.$ 
\end{proof}

\section{Proof of the Theorem \ref{T1}}

We start this section by  showing  that the set of the numbers $t$ for which $(\tilde{P})$ has solution is upper bounded.
\begin{lemma}\label{l2}
	Assume $(G_1)$ and $(P_1)$. Then, there exists a number $\tau_*,$ which does not depend on $f_1\in \mathcal{N}^{\perp},$ such that for all $t>\tau_*$ problem $(\tilde{P})$ has no solution.
\end{lemma}
\begin{proof}
	If $u \in D^{1,2}(\R^N)$ is a solution to $(\tilde{P})$, then
	$$\int_{\R^N}\nabla u \nabla \phi dx=\int_{\R^N}P(x)\left(g(u)+f(x)+t\phi_1\right)\phi dx,\ \forall \phi \in D^{1,2}(\R^N).$$
	Considering $\phi=\phi_1$, we find
	$$\int_{\R^N}\nabla u \nabla \phi_1 dx=\int_{\R^N}P(x)g(u)\phi_1 dx+t$$
and
	$$\int_{\R^N}\lambda_1P(x) u \phi_1 dx=\int_{\R^N}P(x)g(u)\phi_1dx+t.$$
	From (\ref{eq01}) and (\ref{eq02}), 
	$$\int_{\R^N}\lambda_1P(x) u \phi_1 dx\geq\underline{\mu}\int_{\R^N}P(x)u\phi_1 dx-\Theta\int_{\R^N} \phi_1dx +t$$
	and
	$$\int_{\R^N}\lambda_1 P(x)u \phi_1 dx\geq\overline{\mu}\int_{\R^N}P(x)u\phi_1 dx-\Theta\int_{\R^N} \phi_1dx +t,$$
	from where it follows that 
	$$t\leq(\lambda_1-\overline{\mu})\int_{\R^N}P(x)u\phi_1 dx+\Theta\int_{\R^N} \phi_1dx\le \Theta\int_{\R^N} \phi_1dx,\ \mbox{if}\ \int_{\R^N}P(x)u\phi_1 dx\ge0$$
	and
	$$t\leq(\lambda_1-\underline{\mu})\int_{\R^N}P(x)u\phi_1 dx+\Theta\int_{\R^N} \phi_1dx\le \Theta\int_{\R^N} \phi_1dx,\ \mbox{if}\ \int_{\R^N}P(x)u\phi_1 dx\le0.$$
	Thus, the result follows by fixing $\tau_*=C\int_{\R^N} \phi_1dx.$
\end{proof}

In order to show the existence of sub and super-solution to problem $(\tilde{P})$, we will firstly study the existence of solution for the problem below
$$
\left\{
\begin{array}{lcl}
-\Delta w= P(x)(\underline{\mu}w-\Theta+f(x)),\mbox{ in } \mathbb{R}^N\\
\lim_{|x|\to +\infty}w(x)=0,
\end{array}
\right.
\eqno{(P)_{\underline{\mu}}}
$$
where $\underline{\mu}$ and $\Theta>0$ were fixed in (\ref{eq01}) and (\ref{eq02}).
\begin{lemma}\label{l3}
For each $f \in C(\R^N)\cap L^{\infty}(\R^N)$ the problem $(P)_{\underline{\mu}}$ has a unique solution $w \in L^{\infty}(\R ^N)\cap C(\R ^N)$.
\end{lemma}

\begin{proof}
	Using the variational characterization of $\lambda_1$ given in (\ref{lambda1}), we have the inequality below 
	$$\underline{\mu}<\lambda_1=\inf_{v\in D^{1,2}(\R^N)\setminus\{0\}}\left\{\frac{\int_{\R^N}|\nabla v|^2dx}{\int_{\R^N}P(x)|v|^2dx}\right\},$$
	which permits to prove that 
	$$
	\|u\|_{*}=\left[\int_{\R^N}\Big(|\nabla u|^2-\underline{\mu}P(x)|u|^2\Big)dx\right]^{1/2}
	$$ 
	defines a norm in $D^{1,2}(\R^N)$ that is equivalent to usual norm. Then, by Riesz Representation Theorem, there is a solution $w\in D^{1,2}(\R^N)$ of the problem  $(P)_{\underline{\mu}}.$ As 
	$$\underline{\mu}w-\Theta+f(x)\le M(1+|w|)\le M_1(1+|w|^p),$$ 
	for some $p \in(1,2^*-1)$ and $M_1>0,$
	it follows from Lemma \ref{reg} that $w \in L^{\infty}(\R ^N)\cap C(\R ^N)$. Hence, 
	$$P(x)(\underline{\mu}w-\Theta+f(x))\le M_2P(x),\ \mbox{for some constant}\ M_2>0.$$
	From \cite{Alves-Lima-Souto},  
	$$w(x)=\int_{\R^N}\frac{C_NP(y)(\underline{\mu}w(y)-\Theta +f(y))}{|x-y|^{N-2}}dy,$$
	and so,
	$$|w(x)|\le C \left|\int_{\R^N}\frac{P(y)}{|x-y|^{N-2}}dy\right|.$$
	This together with $(P_2)$ gives  
	$$
	|w(x)|\le \frac{C}{|x|^{N-2}}, \quad \forall x \in \mathbb{R}^N \setminus \{0\},
	$$
implying that
	$$\lim_{|x|\to +\infty}w(x)=0,$$
which completes  the proof.	
\end{proof}

As a byproduct of the last lemma and (\ref{eq01}) is the following corollary 

\begin{cor} \label{C1}
The function $w$ obtained in Lemma \ref{l3} is a subsolution for $(\tilde{P})$.
\end{cor}

Our next result establishes a very important relation between the function $w$ obtained in the last lemma and any super-solution of $(\tilde{P})$.

\begin{pro}\label{l4}
	Assume $(G_1)$, $(P_1)$ and $f \in C(\R^N)\cap L^{\infty}(\R^N)$. Then $w \le u$ in $\R^N$, for all possible super-solution $u\in D^{1,2}(\R^N)$ of $(\tilde{P})$.
\end{pro}

\begin{proof}
Since $w$ is a sub-solution and $u$ is a super-solution of $(\tilde{P})$ respectively, we derive that 
	\begin{equation*}\label{m11}
	\int_{\R^N} \nabla(u-w)\nabla\varphi dx\ge\int_{\R^N} \left[P(x)\Big( g(u)+f(x)\Big)- P(x)\Big(\underline{\mu}w-\Theta+f(x)\Big)\right]\varphi dx,
	\end{equation*}
	for all $\varphi\in D^{1,2}(\R^N)$ with $\varphi\ge0$ in $\R^N$. Thereby, by (\ref{eq01}),  
	\begin{equation*}\label{m2}
	\int_{\R^N} \nabla(u-w)\nabla\varphi dx\ge \int_{\R^N} \left[P(x)\Big( \underline{\mu}u-\Theta+f(x)\Big)- P(x)\Big(\underline{\mu}w-\Theta+f(x)\Big)\right]\varphi dx,
	\end{equation*}
	 that is,
	\begin{equation*}\label{m21}
	\int_{\R^N} \nabla(u-w)\nabla\varphi dx\ge\int_{\R^N} \underline{\mu}P(x)(u-w)\varphi dx,
	\end{equation*}
	for all $\varphi\in D^{1,2}(\R^N)$ with $\varphi\ge0$ in $\R^N$. Therefore, 
	\begin{equation*}
	-\Delta (u-w)\ge \underline{\mu} P(x)(u-w),\mbox{ in } \mathbb{R}^N,
	\end{equation*}
and by Lemma \ref{max}, 
	$$u \ge w \quad\mbox{in}\quad\R^N,$$
proving the desired result.		
\end{proof}
As an immediate consequence from Propositions \ref{p2} and \ref{l4}, problem $(\tilde{P})$ has a solution for $f_1 \in \mathcal{N}^{\perp}$ and $t_0 \in \R$ given, then it has also a solution for any $t<t_0.$   

For what has been presented so far, we just need prove the existence of a super-solution to $(P)$.
\begin{lemma}\label{super}
	Let $f_1 \in \mathcal{N}^{\perp}$ be given. Then, there exists $t \in \R$ such that $(\tilde{P})$ has a super-solution.
\end{lemma}
\begin{proof}
	Fix $L>0$ and define
	$$
	m=\max\{g(s)+f_1(x); x \in \R^N\ \mbox{ and }\ 0\le s \le L\}.
	$$
	Now, choose $R_2>R_1>0$ and fix $ F\in C(\mathbb{R}^N,\mathbb{R})$ with $F\equiv 0$ in $B_{R_1}(0)$, $F\equiv m$ in $\R ^N \setminus B_{R_2}(0)$ and $0\le F \le m$ in whole $\R ^N.$  Moreover, let $v$ be the solution to following problem
	$$
	\left\{
	\begin{array}{lcl}
	-\Delta v= P(x)F(x),\mbox{ in } \mathbb{R}^N\\
	\lim_{|x|\to +\infty}v(x)=0.
	\end{array}
	\right. \eqno{(P)_F}
	$$
    As $v$ is the solution to $(P)_F$, we must have
	\begin{equation}\label{r1}
	v(x)=C_N\int_{\R^N}\frac{P(y)F(y)}{|x-y|^{N-2}}dy,
	\end{equation}
	then, 
	$$\begin{array}{lcl}
	|v(x)|\le mC_N\int_{\R^N\setminus B_{R_1}(0)}\frac{P(y)}{|x-y|^{N-2}}dy\vspace{0,4cm}\\
	\,\,\,\,\,\,\,\,\,\,\,\,\le mC_N\int_{(\R^N\setminus B_{R_1}(0))\cap \overline{B}^c_1(x)}\frac{P(y)}{|x-y|^{N-2}}dy+mC_N\int_{(\R^N\setminus B_{R_1}(0))\cap \overline{B}_1(x)}\frac{P(y)}{|x-y|^{N-2}}dy\vspace{0,4cm}\\
	\,\,\,\,\,\,\,\,\,\,\,\,\le mC_N\int_{\R^N\setminus B_{R_1}(0)}P(y)dy+mC_N\int_{(\R^N\setminus B_{R_1}(0))\cap \overline{B}_1(x)}\frac{P(y)}{|x-y|^{N-2}}dy.
	\end{array}$$
	Now, choosing $1<q<\frac{N}{N-2}$, the Hölder's inequality gives 
	$$
	\int_{(\R^N\setminus B_{R_1}(0))\cap \overline{B}_1(x)}\frac{P(y)}{|x-y|^{N-2}}dy\le \Big(\int_{\R^N\setminus B_{R_1}(0)}|P|^{q'}dy\Big)^{\frac{1}{q'}}\Big(\int_{\overline{B}_1(x)}\frac{1}{|x-y|^{q(N-2)}}dy\Big)^\frac{1}{q},
	$$
	where $\frac{1}{q}+\frac{1}{q'}=1$. Therefore, 
	$$
	|v(x)|\le mC_N\int_{\R^N\setminus B_{R_1}(0)}P(y)dy+mC\Big(\int_{\R^N\setminus B_{R_1}(0)}|P|^{q'}dy\Big)^{\frac{1}{q'}}.
	$$
Since $P \in L^{1}(\mathbb{R}^N) \cap L^{q'}(\mathbb{R}^N)$, if we consider $R_1>0$ large enough, we can ensure that 
	$$|v|_{\infty}\le L.$$
From this, we claim that $v$ is a super-solution of $(\tilde{P})$ for a large negative $t$. In fact, choosing a $t$  large negative such that $m+t\phi_1< F$ in whole $\mathbb{R}^N$, we find 
	\begin{equation}\label{r2}
	-\Delta v> P(x)(m+t\phi_1)\geq P(x)(g(v)+t\phi_1+f_1(x)),\quad\mbox{in}\quad\R^N,
	\end{equation}
showing that $v$ is a super-solution for $(\tilde{P})$. Moreover, 	from (\ref{r1}) and 
	$(P_2)$, $$\lim_{|x|\to +\infty}v(x)=0.$$ 
\end{proof}

\subsection*{Proof of Theorem \ref{T1}}

For each $f_1 \in \mathcal{N}^{\perp}$, it follows from the Lemma \ref{super} that there is a $t_0\in \mathbb{R}$ such that problem $(\tilde{P})$ has a super-solution $\overline{u}$. Thus, from Corollary \ref{C1}, for these $f_1$ and $t_0$ the problem $(\tilde{P})$ has a sub-solution $w$ with $w\le \overline{u}.$ Hence, by Proposition \ref{p2}, $(\tilde{P})$ has a solution $u$ satisfying  $w \le u \le \overline{u}.$ The Lemma \ref{l2} ensures that the set of the numbers  $t$ such that $(\tilde{P})$ has a solution is bounded from above, and this set is seen to be a half-line. In order to finish this proof we define $\alpha(f_1)$ to be the supremum  of the numbers $t$ for which problem $(\tilde{P})$ has a solution.

\section{A Priori Estimate}\label{pe}
Now, we will prove the Theorem \ref{T2} via Leray-Schauder Degree Theory. To this end, our first step is to establish a priori estimate for the solutions of $(\tilde{P})$.
\begin{lemma}\label{l11}
		Assume $(G_1)$, $(P_1)$ and $(P_2)$. For each $f \in C(\R^N)\cap L^{\infty}(\R^N)$ the negative part $u^{-}$ of the solutions $u$ of $(\tilde{P})$ are uniformly bounded, that is, there exists a constant $C>0$, independent of $u$, such that $|u^{-}|_{\infty}\leq C$. Moreover, we also have that $|u^-|_{2^*}\leq|w|_{2^*}$, where $w$ was given in Lemma \ref{l3}.
\end{lemma}

\begin{proof}
	By Proposition \ref{l4}, $w\leq u$ in $\mathbb{R}^N$. Fixing $K=|w|_{\infty}$, if $u(x)<0$, we have $-K\leq w(x)\leq u(x)<0$, then $|u^-|\leq |w|$ in $\mathbb{R}^N$. From this,  $|u^-|_{\infty}\leq K$ and $|u^-|_{2^*}\leq|w|_{2^*}$.
\end{proof}

In the sequel, analogous to \cite{Djairo}, without loss of generality  we are assuming that $g$ satisfies 
\begin{equation}
g(s)\geq 0,\quad \mbox{for all}\quad s\geq 0.
\end{equation}

\begin{lemma}\label{lem3.3}
		Assume $(G_1)$, $(P_1)$ and $(P_2)$. For each $f \in C(\R^N)\cap L^{\infty}(\R^N)$ there is a constant $C>0$ such that
		\begin{equation}\label{ineq}
		\left|\int_{\R^N}P(x)g(u^+)\phi_1 dx\right|\leq C
		\end{equation}
		for all solutions $u$ of $(\tilde{P})$, where $\phi_1$ is the positive eigenfunction associated with the first eigenvalue of $(P)_{\lambda}$.
\end{lemma}
\begin{proof}
	Taking $\phi_1$ as a test function in the problem $(P)$, we get
	\begin{equation*}
	\lambda_1\int_{\R^N}P(x)u\phi_1 dx=\int_{\R^N}P(x)g(u)\phi_1 dx+\int_{\R^N}P(x)f(x)\phi_1 dx.
	\end{equation*}
From (\ref{eq01}) and (\ref{eq02}), 
	\begin{equation}
	(\underline{\mu}-\lambda_1)\int_{\R^N}P(x)u\phi_1 dx\leq \Theta\int_{\R^N}P(x)\phi_1dx-\int_{\R^N}P(x)f(x)\phi_1dx\leq C_1
	\end{equation}
	and
	\begin{equation}
	(\overline{\mu}-\lambda_1)\int_{\R^N}P(x)u\phi_1 dx\leq \Theta\int_{\R^N}P(x)\phi_1dx-\int_{\R^N}P(x)f(x)\phi_1dx\leq C_1,
	\end{equation}
implying that $\left|\int_{\R^N}P(x)u\phi_1 dx\right|\leq C_2<+\infty$ for some $C_2>0$. Consequently, $\left|\int_{\R^N}P(x)g(u)\phi_1 dx\right|\leq C_3$. This together with the Lemma \ref{l11} ensures that $\left|\int_{\R^N}P(x)g(u^+)\phi_1 dx\right|\leq C$, for some $C>0$.	
\end{proof}

\begin{lemma}
	Assume $(G_1)$, $(P_1)$ and $(P_2)$. Then for a given $f \in C(\R^N)\cap L^{\infty}(\R^N)$ there is a constant $C>0$ such that
	\begin{equation}
	\int_{\R^N}|\nabla u^+|^2dx\leq C
	\end{equation}
	for all solutions $u$ of $(\tilde{P})$.
\end{lemma}	
	
\begin{proof}
	Given a solution $u$ of $(\tilde{P})$, we know that $u\in D^{1,2}(\R^N)$ and $u^{+}\in D^{1,2}(\R^N)$. Taking $u^{+}$ as a test function in $(\tilde{P})$, we see that 
	\begin{equation}\label{9}
	\int_{\R^N}|\nabla u^{+}|^2dx=\int_{\R^N}P(x)g(u^{+})u^{+}dx+\int_{\R^N}P(x)f(x)u^{+}dx.
	\end{equation}
	The last term in (\ref{9}) gives no problem, since $f\in L^{\infty}(\R^N)$ leads to 
	\begin{equation}
	\left| \int_{\R^N}P(x)f(x)u^{+}dx \right|\leq C|u^+|_{2,P}\leq C|\nabla u^{+}|_{2}.
	\end{equation}
	In order to estimate the first term in the right side of (\ref{9}), we rewrite it of the form
	\begin{equation}
	\int_{\R^N}P(x)g(u^{+})u^{+}dx=\int_{\R^N}\left( P(x)g(u^+)\phi_1\right)^{\beta}\left( P(x)g(u^+)\right)^{1-\beta}\phi_1^{-\beta}u^+dx,
	\end{equation}
	where $\beta=\frac{\sigma-1}{\sigma} \in (0,1)$, with $\sigma$ given in (\ref{sigma}). Thus, using the Hölder inequality for $p=1/\beta$ and $q=1/(1-\beta)$ where $1/p+1/q=1$, we find
	\begin{equation}\label{12}
	\int_{\R^N}P(x)g(u^{+})u^{+}dx\leq \left( \int_{\R^N}P(x)g(u^+)\phi_1dx\right)^{\beta}\left( \int_{\R^N}\frac{P(x)g(u^+)(u^+)^{\frac{1}{1-\beta}}}{\phi_1^{\frac{\beta}{1-\beta}}}dx\right)^{1-\beta}.
	\end{equation}
	On the other hand, we know that the first term in the right side of (\ref{12}) is bounded by Lemma \ref{lem3.3}, thus
	\begin{equation}
	\int_{\R^N}P(x)g(u^{+})u^{+}dx\leq C\left( \int_{\R^N}\frac{P(x)g(u^+)(u^+)^{\frac{1}{1-\beta}}}{\phi_1^{\frac{\beta}{1-\beta}}}dx\right)^{1-\beta}=C\left( \int_{\R^N}\frac{P(x)g(u^+)(u^+)^{\sigma}}{\phi_1^{\gamma}}dx\right)^{1/\sigma},
	\end{equation}
	where $\gamma=\frac{\beta}{1-\beta}=\frac{2}{N-2}$.  Now, as
	\begin{equation*}
	\left( \int_{\R^N}\frac{P(x)g(u^+)(u^+)^{\sigma}}{\phi_1^{\gamma}}dx\right)^{1/\sigma}\leq C\left( \int_{\overline{B}^c_R}\frac{P(x)g(u^+)(u^+)^{\sigma}}{\phi_1^{\gamma}}dx\right)^{1/\sigma}+C\left( \int_{\overline{B}_R}\frac{P(x)g(u^+)(u^+)^{\sigma}}{\phi_1^{\gamma}}dx\right)^{1/\sigma}
	\end{equation*}	
the condition (\ref{m1}) gives 
	\begin{equation}
		\left( \int_{\R^N}\frac{P(x)g(u^+)(u^+)^{\sigma}}{\phi_1^{\gamma}}dx\right)^{1/\sigma} \leq C\left( \int_{\overline{B}^c_R}P(x)|x|^2g(u^+)(u^+)^{\sigma}dx\right)^{1/\sigma}+C\left( \int_{\overline{B}_R}P(x)g(u^+)(u^+)^{\sigma}dx\right)^{1/\sigma}.
	\end{equation}
Using $(P_1)$ and Hölder inequality, we infer that
	\begin{eqnarray*}
	\left( \int_{\overline{B}^c_R}P(x)|x|^2g(u^+)(u^+)^{\sigma}dx\right)^{1/\sigma}&\leq&\left( \int_{\R^N}P(x)|x|^2g(u^+)(u^+)^{\sigma}dx\right)^{1/\sigma}\\\nonumber
	&\leq&C\varepsilon \left( \int_{\R^N}P(x)|x|^2(u^+)^{2\sigma}dx\right)^{1/\sigma}+C\left( \int_{\R^N}P(x)|x|^2(u^+)^{\sigma}dx\right)^{1/\sigma}\\\nonumber
	&\leq&C\varepsilon \left( \int_{\R^N}(u^+)^{2^*}dx\right)^{2/2^*}+C\left( \int_{\R^N}P(x)|x|^2(u^+)^{\sigma}dx\right)^{1/\sigma}\\\nonumber
	&\leq& C\varepsilon|\nabla u^+|_2^2+C\left( \int_{\R^N}P(x)|x|^2(u^+)^{\sigma}dx\right)^{1/\sigma}\\\nonumber
	&\leq& C\varepsilon|\nabla u^+|_2^2+C\left( \int_{\R^N}(u^+)^{2\sigma}dx\right)^{1/2\sigma}\\\nonumber
	&\leq& C\varepsilon|\nabla u^+|_2^2+C|u^+|_{2^*}
	\end{eqnarray*}
	implying that
	\begin{equation}
	\left( \int_{\overline{B}^c_R}P(x)|x|^2g(u^+)(u^+)^{\sigma}dx\right)^{1/\sigma}\leq C\varepsilon|\nabla u^+|_2^2+C|\nabla u^+|_2.
	\end{equation}
Moreover,
	\begin{eqnarray*}
	\left( \int_{\overline{B}_R}P(x)g(u^+)(u^+)^{\sigma}dx\right)^{1/\sigma}&\leq& \left( \int_{\R^N}P(x)g(u^+)(u^+)^{\sigma}dx\right)^{1/\sigma}\\\nonumber
	&\leq& \left( \int_{\R^N}P(x)[\varepsilon(u^+)^{\sigma}+C](u^+)^{\sigma}dx\right)^{1/\sigma}\\\nonumber
	&\leq& C\varepsilon\left( \int_{\R^N}P(x)(u^+)^{2\sigma}dx\right)^{1/\sigma}+C\left( \int_{\R^N}P(x)(u^+)^{\sigma}dx\right)^{1/\sigma}\\\nonumber
	&\leq& C\varepsilon|u^+|^2_{2^*}+C|u^+|_{2^*}\\\nonumber
	&\leq& C\varepsilon|\nabla u^+|^2_{2}+C|\nabla u^+|_{2}\nonumber
	\end{eqnarray*}
	that is,
	\begin{equation}
	\left( \int_{\overline{B}_R}P(x)g(u^+)(u^+)^{\sigma}dx\right)^{1/\sigma}\leq C\varepsilon|\nabla u^+|^2_{2}+C|\nabla u^+|_{2}.
	\end{equation}
	Therefore, from $(20)-(27)$, 
	\begin{equation*}
	|\nabla u^+|_2^2\leq C\varepsilon|\nabla u^+|_2^2+C|\nabla u^+|_2
	\end{equation*}
or equivalently
	\begin{equation*}
	(1-C\varepsilon)|\nabla u^+|_2^2\leq C|\nabla u^+|_2.
	\end{equation*}
This finishes the proof.	
\end{proof}
	
\begin{lemma} \label{3.4} Assume $(G_1)$, $(P_1)$ and $(P_2)$. Then for each $f \in C(\R^N)\cap L^{\infty}(\R^N)$ there is a constant $C>0$ that  depends on the norm of $f$ in $L^{\infty}(\R^N)$ such that $|u|_{\infty}\leq C$, for all solutions $u$ of $(P)$.
\end{lemma}	

\begin{proof}
See proof of Lemma \ref{reg}.
\end{proof}

\section{A solution operator}

First of all, as the embedding  $D^{1,2}(\R^N)\hookrightarrow L^2_P(\R^N)$ is compact  (see \cite{Edelson-Rumbos1}), we can guarantee that for each Let $v\in L^2_P(\R^N)$ the linear functional $\Psi:D^{1,2}(\R^N)\rightarrow\R$ given by
\begin{equation*}
\Psi(\varphi):=\int_{\R^N}P(x)[g(v)+t\phi_1+f_1]\varphi dx
\end{equation*}
is continuous. Indeed, note that
\begin{eqnarray*}
|\Psi(\varphi)|&\leq& \int_{\R^N}P(x)[|g(v)|+|t|\phi_1+|f_1|]\cdot |\varphi| dx\\
&=&\int_{\R^N} P(x)|g(v)|\cdot |\varphi| dx+|t|\int_{\R^N} P(x)\phi_1\cdot |\varphi| dx+ \int_{\R^N} P(x)f_1|\varphi| dx.
\end{eqnarray*}
By Hölder inequality, 
\begin{equation*}
\int_{\R^N} P(x)\phi_1\cdot |\varphi| dx=\int_{\R^N} P^{1/2}(x)\phi_1\cdot P^{1/2}(x)|\varphi| dx\leq |\phi_1|_{2,P}|\varphi|_{2,P}\leq C|\varphi|_{1,2}
\end{equation*}
and
\begin{equation*}
\int_{\R^N} P(x)f_1|\varphi| dx\leq |f_1|_{\infty}\int_{\R^N}P^{1/2}(x)P^{1/2}(x)\varphi dx\leq C|\varphi|_{2,P}\leq C|\varphi|_{1,2}.
\end{equation*}
On the other hand, as $P^{1/2}g(v)\in L^2(\R^N)$ and $P^{1/2}|\varphi|\in L^2(\R^N)$, we also have
\begin{equation*}
\int_{\R^N} P(x)|g(v)|\cdot |\varphi| dx\leq C|\varphi|_{2,P}\leq C|\varphi|_{1,2}.
\end{equation*}
Using the estimates above we can infer that
\begin{equation*}
|\Psi(\varphi)|\leq C|\varphi|_{1,2},\quad\forall \varphi\in D^{1,2}(\R^N), 
\end{equation*}
for some positive constant $C>0$. This shows that  $\Psi$ is continuous. 

Using Riesz's Theorem, there exists a unique $u\in D^{1,2}(\R^N)$ such that
\begin{equation*}
\int_{\R^N}\nabla u\nabla\varphi dx= \int_{\R^N}P(x)[g(v)+t\phi_1+f_1]\varphi dx,\quad \forall \varphi\in D^{1,2}(\R^N).
\end{equation*}

Therefore, it is well defined the linear solution operator $K_t:L^2_P(\R^N)\rightarrow L^2_P(\R^N)$, which is a compact operator, given by $K_t(v):=u$ where $u$ is the unique solution of the problem
\begin{equation*}
-\Delta u=P(x)\Big( g(v)+t\phi_1(x)+f_1(x)\Big),\mbox{ in } \mathbb{R}^N,\quad u\in D^{1,2}(\R^N).
\end{equation*}
Observe that the solutions of the Ambrosetti-Prodi type problem below 
\begin{equation*}
-\Delta u=P(x)\Big( g(u)+t\phi_1(x)+f_1(x)\Big),\mbox{ in } \mathbb{R}^N,\quad u\in D^{1,2}(\R^N)
\end{equation*}
are fixed points of the operator $K_t$.

Now, as in \cite{Alves-Lima-Souto}, we will consider the space
$$
E_0=\left\{v \in C(\R^N,\R);\ \sup_{x \in \mathbb{R}^N}\left(|x|^{N-2}|v(x)|\right)<+\infty\right\}
$$
equipped with the norm
$$
\|v\|=\sup_{x \in \mathbb{R^N}}|v(x)|+\sup_{x \in \mathbb{R}^N}\left(|x|^{N-2}|v(x)|\right)
$$
that is, 
$$
\|v\|=|v|_{\infty}+|(|\cdot|^{N-2})v|_{\infty}.
$$
A simple computation gives that $(E_0, \|\cdot\|)$ is a Banach space and the embedding $E_0 \hookrightarrow L^{2}_P(\R^N)$ is continuous.
From this, given for each $v \in E_0$ the problem
$$
\left\{
\begin{array}{lcl}
-\Delta u=P(x)\Big( g(v)+t\phi_1(x)+f_1(x)\Big),\mbox{ in } \mathbb{R}^N\\
u \in D^{1,2}(\R^N).
\end{array}
\right.
\eqno{(WLP)}
$$
has a unique weak solution $u=:K_t(v)\in D^{1,2}(\R^N) \hookrightarrow L^2_P(\R^N)$. The lemma below shows that $u\in E_0$. 
\begin{lemma} \label{continuidade} $K_t(E_0) \subset E_0$.
\end{lemma}

\begin{proof}
	Let $v \in E_0$ and $u=K_t(v) \in D^{1,2}(\R^N)$. Then,
	\begin{equation}\label{repre}
	u(x)=C\int_{\mathbb{R}^N}\frac{P(y)[g(v)+mv+t\phi_1(y)+f_1(y)]}{|x-y|^{N-2}}dy, \quad \forall x \in \R^N.
	\end{equation}
	First of all we are going to prove that $u$ is a continuous function in $\R^N$. Fixed $x_0\in \R^N$, for all $y \in \R^N$ it holds
	$$
	|u(y)-u(x_0)| \leq C\int_{\R^N}\left|\frac{P(z)}{|y-z|^{N-2}}-\frac{P(z)}{|x_0-z|^{N-2}}\right|dz .
	$$
	Given $\delta_1>0$ and $y \in B_{\delta_1/2}(x_0)$, we have $ B_{\delta_1}(x_0) \subset B_{2\delta_1}(y)$ and 
	$$\begin{array}{ll}
\ddd\int_{B_{\delta_1}(x_0)}\left|\frac{P(z)}{|y-z|^{N-2}}-\frac{P(z)}{|x_0-z|^{N-2}}\right|dz &\leq C_1 \ddd\int_{B_{\delta_1}(x_0)}\frac{P(z)}{|x_0-z|^{N-2}}dz \\ 
	& +\ddd\int_{B_{2\delta_1}(y)}\frac{P(z)}{|y-z|^{N-2}}dz\\
	&\leq C_1|P|_{\infty}\delta_1^{2}+C_2|P|_{\infty}(2\delta_1)^{2}.
	\end{array}
	$$
	From this, given $\varepsilon>0$, let us fix $\delta_1>0$ verifying 
	$$
	C_1|P|_{\infty}\delta_1^{2}+C_2|P|_{\infty}(2\delta_1)^{2}<\epsilon/2,
	$$
	consequently
	\begin{equation} \label{E1}
	\left|\int_{B_{\delta_1}(x_0)}\left(\frac{P(z)}{|y-z|^{N-2}}-\frac{P(z)}{|x_0-z|^{N-2}}\right)dz \right|<\frac{\epsilon}{2}.
	\end{equation}
	On the other hand, for  $y \in B_{\delta_1/2}(x_0)$ and $z \in \mathbb{R}^N\setminus B_{\delta_1}(x_0)$, we derive
	$$
	|y-z|\ge |z-x_0|-|x_0-y|\ge \frac{\delta_1}{2},
	$$
	then
	$$\frac{1}{|y-z|^{N-2}}\le \left(\frac{2}{\delta_1}\right)^{N-2}$$
	and
	$$\frac{1}{|x_0-z|^{N-2}}\le \left(\frac{1}{\delta_1}\right)^{N-2}.$$
	Hence,
	$$
	\left|\frac{P(z)}{|y-z|^{N-2}}-\frac{P(z)}{|x_0-z|^{N-2}}\right|\le C P(z),\, \forall z \in \mathbb{R}^N\setminus B_{\delta_1}(x_0)\ \mbox{and}\ \forall y \in B_{\delta_1/2}(x_0).
	$$
	From Lebesgue's Theorem 
	$$
	\lim_{y\to x_0}\int_{\mathbb{R}^N\setminus B_{\delta_1}(x_0)}\left|\frac{P(z)}{|y-z|^{N-2}}-\frac{P(z)}{|x_0-z|^{N-2}}\right|dz=0.
	$$
	The last limit together with (\ref{E1}) implies that there exists $\delta \in (0, \delta_1/2)$ such that 
	$$
	\int_{\mathbb{R}^N}\left|\frac{P(z)}{|y-z|^{N-2}}-\frac{P(z)}{|x_0-z|^{N-2}}\right|dz< \frac{\epsilon}{2} \quad \mbox{for} \quad |y-x_0| < \delta,
	$$
	showing that $u$ is continuous at $x_0$. As $x_0$ is arbitrary, $u$ is continuous in $\R^N$. In order to conclude that $u \in E_0$, it is enough to recall  that
	$$
	|u(x)|\le C\int_{\mathbb{R}^N}\frac{P(y)}{|x-y|^{N-2}}dy,
	$$
	because this inequality combined withe $(P_2)$ leads to  
	$$
	|u(x)| \leq \frac{C}{|x|^{N-2}}, \quad \forall x \in \R^N \setminus \{0\},
	$$
	for some $C>0$, proving the desired result.
\end{proof}

The last lemma permits to consider the operator $T_t:=K_t|_{E_0}:E_0 \to E_0$. Our next step is to show that $T_t$ is a compact operator, which is a crucial property to apply our approach.   

\begin{lemma}
	The solution operator $T_t:E_0 \to E_0$ is a compact linear operator. 
\end{lemma}

\begin{proof}
The proof follows as in \cite[Lemma 3.3]{Alves-delima-Nobrega}. 
\end{proof}

We observe that the solutions of our problem $(\tilde{P})$ are the zeros of $I-T_t$. In the next section, we will proceed to compute the degree of $I-T_t$ in subsets of $E_0$.

\section{Computation of Some Topological Degrees}

\begin{lemma}\label{3.7}
		Assume $(G_1)$, $(P_1)$ and $(P_2)$. Let $f_1\in \mathcal{N}^{\perp}$ and $t_0\in\R$. Then, there is $R>0$ such that $deg(I-T_{t_0},B_R,0)=0$, where $B_R=\{u\in E_0; \|u\|<R\}$.
\end{lemma}

\begin{proof}
	From Theorem \ref{T1}, the problem $(\tilde{P})$ has no solution if $t=t_1>\alpha(f_1)$. Thus, choosing a $t_1>\max\{\alpha(f_1),t_0\}$, there is a constant $R>0$ such that $\|u\|<R$, for all possible solutions of $(\tilde{P})$ with $f_1\in \mathcal{N}^{\perp}$ fixed and all $t\in[t_0,t_1]$, see Lemma \ref{3.4}. Clearly, the operator $T_t$, for $t\in [t_0,t_1]$, is a compact homotopy joining $T_{t_0}$ to $T_{t_1}$. As $(I-T_t)(u)\neq0$ always that $\|u\|=R$ and $t\in[t_0,t_1]$, by topological degree we have the equality  $deg(I-T_{t_0},B_R,0)=deg(I-T_{t_1},B_R,0)=0$, since problem $(\tilde{P})$ with $t=t_1$ has no solution.
\end{proof}

\begin{lemma}
	Assume $(G_1)$, $(P_1)$ and $(P_2)$. Given $f_1\in \mathcal{N}^{\perp}$ and $t_0<\alpha(f_1)$, there is an open bounded subset $\mathcal{O}$ of $E_0$, such that $deg(I-T_{t_0},\mathcal{O},0)=1$.
\end{lemma}

\begin{proof}
	Let $t_1\in\R$ such that $t_0<t_1<\alpha(f_1)$. From Theorem \ref{T1}, we know that $(\tilde{P})$ with $t=t_1$ has a solution $\overline{u}$, which is a supersolution of $(\tilde{P})$ with $t=t_0$, that is, 
	\begin{equation*}\label{30}
	\left\{
	\begin{array}{lcl}
	-\Delta \overline{u}>P(x)\left(g(\overline{u})+t_0\phi_1(x)+f_1(x)\right),\quad\R^N\\
	\lim_{|x|\to +\infty}\overline{u}(x)=0.
	\end{array}
	\right.
	\end{equation*}
By Proposition (\ref{l4}), there is a subsolution $\underline{u}$ of 
	\begin{equation*}
	\left\{
	\begin{array}{lcl}
	-\Delta \underline{u}<P(x)\left(g(\underline{u})+t_0\phi_1(x)+f_1(x)\right),\quad\R^N\\
	\lim_{|x|\to +\infty}\underline{u}(x)=0
	\end{array}
	\right.
	\end{equation*}
that satisfies $\underline{u}<\overline{u}$ in $\R^N$. Now, let us consider the set
	\begin{equation}
	\mathcal{O}=\{u \in B_r(0) \subset E_0;\, \underline{u}(x)<u(x)<\overline{u}(x)\mbox{ in }\R^N,\,\liminf_{|x|\rightarrow\infty}|x|^{N-2}(u-\underline{u})>0,\quad \liminf_{|x|\rightarrow\infty}|x|^{N-2}(\overline{u}-u)>0\}
	\end{equation}
	where the radius $r$ will be determined later on. 
	\begin{claim}
		The set $\mathcal{O}$ is open, bounded and convex.
	\end{claim}
Since the boundedness and convexity of $\mathcal{O}$ are immediate, we will only prove that $\mathcal{O}$ is open. Indeed, for each $u\in \mathcal{O}$, there is $\xi>0$ such that 
	\begin{equation}\label{b1}
	\liminf_{|x|\rightarrow\infty}|x|^{N-2}(u-\underline{u}),\liminf_{|x|\rightarrow\infty}|x|^{N-2}(\overline{u}-u)\geq \xi>0.
	\end{equation}
Therefore, for $R$ large enough, 
\begin{equation}\label{b3}
	|x|^{N-2}(u-\underline{u}),|x|^{N-2}(\overline{u}-u)\geq \xi/2,\quad |x|\geq R.
\end{equation}
Since
	\begin{equation}\label{b2}
	\underline{u}(x)<u(x)<\overline{u}(x),\quad\forall x\in\R^N,
	\end{equation}
for $\delta \in (0,\xi/2)$ small enough, we must have $\underline{u}(x)<z(x)<\overline{u}(x)$ for $|x|\leq R$ with $z\in B_{\delta}(u) \subset E_0$. On the other hand, for $|x|\geq R$,
	\begin{equation}\label{b4}
	|x|^{N-2}[z(x)-\underline{u}(x)]=|x|^{N-2}[z(x)-u(x)]+|x|^{N-2}[u(x)-\underline{u}(x)]\geq -\delta+\xi/2>0,
	\end{equation}
	similarly,
	\begin{equation}\label{b5}
	|x|^{N-2}[\overline{u}(x)-z(x)]\geq -\delta+\xi/2>0.
	\end{equation}
	Consequently, from (\ref{b1})-(\ref{b5}), $z\in \mathcal{O}$, and so, $B_{\delta}(u)\subset\mathcal{O}$. This proves that $\mathcal{O}$ is open.
	
	Now, define $\overline{g}:\R^N\times \R\rightarrow\R$ by
	$$
	\overline{g}(x,s)=\left\{
	\begin{array}{lcl}
	g(\underline{u}(x)),\quad s\leq \underline{u}(x),\\
	g(s),\quad \underline{u}(x)\leq s\leq\overline{u}(x),\\
	g(\overline{u}(x)),\quad s\geq \overline{u}(x).
	\end{array}
	\right.
	$$
	Recalling that $g$ is increasing and $\overline{u}\in L^{\infty}(\R^N)$, it follows that $\overline{g}$ is bounded in $\R^N\times \R$ and increasing in the variable $s$. Now, by the same arguments used in the previous sections, for each $v\in E_0$, there is a unique solution $u\in E_0$ of the problem below
	$$
	\left\{
	\begin{array}{lcl}
	-\Delta u=P(x)\Big( \overline{g}(x,v)+t\phi_1(x)+f_1(x)\Big),\mbox{ in } \mathbb{R}^N\\
	u \in D^{1,2}(\R^N),\quad \lim_{|x| \to +\infty}u(x)=0.
	\end{array}
	\right.
	\eqno{(WLQ)}
	$$
	Thus, we can defined a compact operator $\overline{T}_t:E_0\rightarrow E_0$ such that $\overline{T}_t(v)=u$. From definition of $\overline{T}_t$, we see that $\overline{T}_t=T_t$ in $\overline{\mathcal{O}}$. 
	\begin{claim}
	The application $\overline{T}_t$ maps $E_0$ into $\mathcal{O}$, provided $r$ is properly chosen.
	\end{claim}
	Since $\overline{g}$ is bounded, from Lemma \ref{continuidade} the solutions $u$ of $(WLQ)$ are uniformly bounded in $E_0$, for all $v\in E_0$. From now on, let us fix  $r>\sup\{\|\overline{T}_{t_0}v\|; v\in E_0\}$. Our goal is proving that for each $v\in E_0$ we have that $u=\overline{T}_{t_0}v \in \mathcal{O}$. Obviously, $\|u\|<r$.  From Lemma \ref{super} and $(WLQ)$, there are $\overline{u} , u \in D^{1,2}(\mathbb{R}^N)$ satisfying  
	\begin{equation*}
	\left\{
	\begin{array}{lcl}
	-\Delta\overline{u}=P(x)F(x)>P(x)[g(\overline{u})+t_0\phi_1+f_1],\quad\R^N\\
	 \overline{u}\in D^{1,2}(\R^N),\quad\lim_{|x|\to +\infty}\overline{u}(x)=0
	\end{array}
	\right.
	\end{equation*}
	and
	\begin{equation*}
	\left\{
	\begin{array}{lcl}
	-\Delta u=P(x)[\overline{g}(x,v)+t_0\phi_1+f_1],\quad\R^N\\
		u \in D^{1,2}(\R^N),\quad\lim_{|x|\to +\infty}u(x)=0.
	\end{array}
	\right.
	\end{equation*}
Then,
	\begin{equation*}
	\left\{
	\begin{array}{lcl}
	(-\Delta)(\overline{u}-u)=P(x)[F(x)-\overline{g}(x,v)+t_0\phi_1+f_1],\quad\R^N\\
	(\overline{u}-u)\in D^{1,2}(\R^N),\quad\lim_{|x|\to +\infty}(\overline{u}-u)(x)=0.
	\end{array}
	\right.
	\end{equation*}
Thereby, by Riesz representation,
	\begin{equation*}
	(\overline{u}-u)(x)=\int_{\R^N}\frac{P(y)[F(y)-\overline{g}(y,v)+t_0\phi_1+f_1]}{|x-y|^{N-2}}dy> \int_{\R^N}\frac{P(y)[g(\overline{u})-\overline{g}(y,v)]}{|x-y|^{N-2}}dy\ge0.
	\end{equation*}
	Therefore, by \cite[Lemma 3.1]{Alves-Lima-Souto}, 
	\begin{equation*}
	\overline{u}(x)>u(x),\quad \forall x\in\R^N\mbox{ and }\liminf_{|x| \to \infty}|x|^{N-2}(\overline{u}-u)(x)>0.
	\end{equation*}
	Similarly, 
	\begin{equation*}
	u(x)>\underline{u}(x),\quad \forall x\in\R^N\mbox{ and }\liminf_{|x| \to \infty}|x|^{N-2}(u-\underline{u})(x)>0.
	\end{equation*}
This shows that $\overline{T}_t(E_0)\subset \mathcal{O}$.

	Now, fixed  $\psi\in \mathcal{O}$, let us consider the compact homotopy $H_{\theta}(u)=\theta\overline{T}_{t_0}(u)+(1-\theta)\psi$, $0\leq \theta\leq 1$ .As $u\neq H_{\theta}(u)$ for all $u\in \partial \mathcal{O}$ and all $\theta\in[0,1]$, we conclude that $deg(I-H_1,\mathcal{O},0)=deg(I-H_0,\mathcal{O},0)$. However, $H_0$ is a constant operator and clearly $deg(I-H_0,\mathcal{O},0)=1$. So $deg(I-\overline{T}_{t_0},\mathcal{O},0)=1$, and the lemma is proved.	
\end{proof}

\subsection*{Proof of Theorem \ref{T2}}

\begin{proof}
	
\textbf{(i)} For $f_1\in\mathcal{N}^{\perp}$ and $t_0<\alpha(f_1)$ be given. By previous lemma, there is a bounded open set $\mathcal{O}$ such that
\begin{equation*}
deg(I-T_{t_0},\mathcal{O},0)=1.
\end{equation*}
Therefore $I-T_{t_0}$ has a zero in $\mathcal{O}$, that is, problem $(\tilde{P})$ has a solution $u_1\in \mathcal{O}$, for these given $f_1$ and $t_0$. Now, choose $R>0$ such that $R>r$. Since by Lemma \ref{3.7}  $deg(I-T_{t_0},B_R,0)=0$, it follows that $deg(I-T_{t_0},B_R\setminus\overline{\mathcal{O}},0)=-1$. This proves that $(\tilde{P})$ has another solution $u_2\in B_R\setminus\overline{\mathcal{O}}$.

\textbf{(ii)} Now, let $f_1\in\mathcal{N}^{\perp}$ and $t_0=\alpha(f_1)$ be given. Take a sequence $t_n<t_0$ and $t_n\rightarrow t_0$. It follows from Theorem \ref{T1} that $(\tilde{P})$ has a solution $u_n$ for the given function $f_1$ and each $t_n$. From the results of Section \ref{pe}, the sequence $(u_n)$ is bounded in $D^{1,2}(\R^N)$. Thus, since $D^{1,2}(\R^N)$ is compactly embedded in $L^2_P(\R^N)$, we have that there is $u\in D^{1,2}(\R^N)$ such that $u_n\rightharpoonup u$ in $D^{1,2}(\R^N)$ and $u_n\rightarrow u$ in $L^2_P(\R^N)$. Therefore, $u$ is a solution of problem
\begin{equation*}
-\Delta u=P(x)\Big( g(u)+t_0\phi_1(x)+f_1(x)\Big),\mbox{ in } \mathbb{R}^N,\quad u\in D^{1,2}(\R^N).
\end{equation*}
Using the same argument of the Lemma \ref{continuidade}, we also $u\in E_0$, that is, 
$$
\lim_{|x| \rightarrow \infty}u(x)=0,
$$ 
finishing the proof. 
\end{proof}


\begin{thebibliography}{99}
	
\addcontentsline{toc}{section}{References}

\bibitem{Aizicovici-Papageorgiou-Staicu}S. Aizicovici, N. S. Papageorgiou and V. Staicu, \textit{Sublinear and superlinear Ambrosetti-Prodi problems for the Dirichlet p-Laplacian}, Nonlinear Anal. \textbf{95} (2014) 263-280.

\bibitem{Alves-delima-Nobrega}C. O. Alves, R. N. de Lima and A. B. Nóbrega, \textit{Bifurcation properties for a class of fractional Laplacian equations in $\R^N$}, Mathematische Nachrichten, \textbf{291} 2125-2144 (2018). 

\bibitem{Alves-delima-Nobrega2}C. O. Alves, R. N. de Lima and A. B. Nóbrega, \textit{Global bifurcation results for a fractional elliptic equation in $\R^N$}, To appear in J. Math. Anal. Appl. https://doi.org/10.1016/j.jmaa.2020.123980

\bibitem{Alves-Lima-Souto} C. O. Alves, R. N. de Lima and M. A. S. Souto, \textit{Existence of a solution for a non-local problem in $\R^N$ via bifurcation theory}, Proc. Edin. Math. Soc., \textbf{61} , 825-845 (2018).



\bibitem{Ambrosetti-Prodi} A. Ambrosetti and G. Prodi, \textit{On the inversion of some differentiable mappings with singularities between Banach spaces}, Ann. Mat. Pura Appl. \textbf{93} (1972), 231-246.

\bibitem{Arcoya-Ruiz} D. Arcoya and D. Ruiz, \textit{The Ambrosetti-Prodi problem for the p-Laplace operator},
Comm. in Part. Diff. Equations 31 (2006) 849-290.

\bibitem{Berger-Podolak} M.S. Berger and E. Podolak, \textit{On the solutions of a nonlinear Dirichlet problem}, Indiana Univ. Math. J. \textbf{24} (1974/1975) 837-846.

\bibitem{Brezis-Turner} H. Brézis and R.E.L. Turner, \textit{On a class of superlinear elliptic problems}, Comm. in Part. Diff. Equations \textbf{2} (1977) 601-614.

\bibitem{de Morais Filho}D.C. de Morais Filho, \textit{A variational approach to an Ambrosetti-Prodi type problem for a system of elliptic equations}, Nonlinear. Anal. \textbf{26} (10) (1996) 1655-1668.

\bibitem{Djairo} D. G. de Figueiredo, {\it Lectures on boundary value problems of the Ambrosetti-Prodi type}, in: Atas do 12 Seminario Brasileiro de Analyse (12th Brazilian Analysis Seminar), Universidade de São Paulo, Sao Paulo, 1980, pp. 230-292.

\bibitem{djairo2}D.G. de Figueiredo and S. Solimini, \textit{A variational approach to superlinear elliptic problems}, Comm. in Part. Diff. Equations \textbf{9} (7) (1984) 699-717.

\bibitem{djairo 3} D.G. de Figueiredo and J. Yang, \textit{Critical superlinear Ambrosetti-Prodi problems}, Topol. Methods Nonlinear Anal. \textbf{14} (1999) 59-80.

\bibitem{djairo 4} D. G. de Figueiredo and B. Sirakov, \textit{On the Ambrosetti-Prodi problem for non-variational elliptic systems} J. Differential Equations \textbf{240} (2007), 357-374.

\bibitem{Edelson-Rumbos1} A. L. Edelson and A. J. Rumbos, \textit{Linear and semilinear eigenvalue problems in $\R^N$}, Comm. in Part. Diff. Equations, \textbf{18}(1-2), 215-240 (1993).



\bibitem{G-T} D. Gilbarg and  N.S. Trudinger, {\em Elliptic partial differential equations of second order.} Reprint of the 1998 edition. Classics in Mathematics. Springer-Verlag, Berlin, 2001.

\bibitem{M} J. Mawhin, \textit{Ambrosetti-Prodi Type Results in Nonlinear Boundary Value Problems}, in Differential Equations and Mathematical Physics, Lecture Notes in Math., Vol. 1285 (Springer, 1987), pp. 290-313.

\bibitem{PM}F. O. de Paiva and M. Montenegro, {\em The Ambrosetti-Prodi type result to a quasilinear Neumann problem.} Proc. Edin. Math. Soc., \textbf{55} , 771-780 (2012).

\bibitem{Presoto}A. E. Presoto and F. O. de Paiva, \textit{A Neumann problem of Ambrosetti-Prodi-type}, J. Fixed Point Theory Appl. \textbf{18} (2016) 189-200.

\bibitem{Rabinowitz}P. H. Rabinowitz, \textit{A note on Semilinear Elliptic Equation on $\R^N$}, Nonlinear Analysis: A Tribute in Honour of G. Prodi, Quad. Scu. Norm Super. Pisa  (1991) 307-318.




	
\end{thebibliography}
\end{document}